\documentclass{alggeom}
\usepackage{amsmath, amsthm}
\usepackage{enumerate,amssymb,setspace}
\usepackage{mathtools, verbatim}
\usepackage{color}
\usepackage[all]{xy}

\renewcommand{\geq}{\geqslant}
\renewcommand{\leq}{\leqslant}
\newcommand{\SL}{\mathrm{SL}}

\newcommand{\diag}{\mathrm{Diag}}

\newcommand{\Pic}{\mathrm{Pic}}
\newcommand{\Proj}{\mathrm{Proj}}

\newcommand{\lcm}{\mathrm{lcm}}

\newtheorem{theorem}{Theorem}[section]
\newtheorem{lemma}[theorem]{Lemma}
\newtheorem{corollary}[theorem]{Corollary}
\newtheorem{proposition}[theorem]{Proposition}

\theoremstyle{definition}
\newtheorem{definition}[theorem]{Definition}

\newtheorem{remark}[theorem]{Remark}
 
\begin{document}

\title[Variations of GIT quotients for pairs]{Variations of geometric invariant quotients for pairs, a computational approach}
\author{Patricio Gallardo}
\email{pgallardocandela@wustl.edu}
\address{Department of Mathematics\\
					Washington University\\
					Campus Box 1146\\
					One Brookings Drive\\
					St. Louis, MO 63130-4899\\
							USA}

\author{Jesus Martinez-Garcia}
\email{J.Martinez.Garcia@bath.ac.uk}
\address{Department of Mathematical Sciences\\
					University of Bath\\
					Bath BA2 7AY
					United Kingdom}
\subjclass[2010]{14L24, 14H10, 14Q10 (Primary) , 14J45, 14J32  (Secondary)}

\date{September 05, 2017}
\begin{abstract} 
We study the GIT compactifications of  pairs formed by a hypersurface and a hyperplane.  We provide a general setting to characterize all polarizations which give rise to different GIT quotients. Furthermore, we describe a finite set of one-parameter subgroups sufficient to determine the stability of any GIT quotient. We characterize all maximal orbits of non stable and strictly semistable pairs, as well as minimal closed orbits of strictly semistable pairs. Our construction gives natural compactifications of the space of log smooth pairs for Fano and Calabi-Yau hypersurfaces.
\end{abstract}

\maketitle

\section{Introduction}
The construction and study of moduli spaces is a central subject in algebraic geometry and Geometric Invariant Theory (GIT) is one of its foundational tools.  It has been applied to study hypersurfaces \cite{gallardo2013git, allcock-cubic-3folds, laza-git-cubic4-folds}; and it is a first step towards constructing  the moduli space of del Pezzo surfaces admitting a K\"ahler--Einstein metric \cite{odaka-spotti-sun}.  A GIT quotient depends on a choice of a line bundle in a parameter space; and any two compactifications, with the exceptions of some limit cases, are related by birational transformations (see \cite{thaddeus-vgit}, \cite{dolgachev-hu-vgit}).

In this article, we consider the GIT quotients parameterizing pairs $(X,H)$ where $X\subset\mathbb P^{n+1}$ is a hypersurface of degree $d$ and $H\subset\mathbb P^{n+1}$ is a hyperplane. This is a natural setting to consider pairs $(X,D)$ where $D=X\cap H$ is a hyperplane section. Our work  generalizes to higher dimensions R. Laza's work on curves \cite{Laza-deformations-variations-git}. Our setting can be automatized to perform computations for any dimension $n$ and degree $d$. Indeed, in the companion to this article \cite{gallardo-jmg-experimental} we provide algorithms, already fully implemented in software \cite{gallardo-jmg-code}, to compute all the invariants and functions in this article. In \cite{gallardo-jmg-vgit-cubic-surfaces}, we apply the current setting and a specific analysis of singularities to describe \emph{geometrically} all GIT compactifications of pairs $(S,C)$ where $S$ is a cubic surface and $C\in |-K_S|$ is an anticanonical divisor.

Let $\mathcal R_{n,d}$ be the parameter space of pairs $(X,H)$. There is a one-dimensional space of stability conditions parametrized by $t \in [0,t_{n,d}]$ corresponding to polarizations of $\mathcal R_{n,d}$ (see Section \ref{sec:setting}). There is a finite number of values $t_i\in \mathbb Q_{\geqslant 0}$ known as \emph{GIT walls} where $0=t_0<t_1<\cdots<t_{n,d}$ and  segments $(t_i,t_{i+1})$  known as \emph{GIT chambers}.  Two GIT quotients are isomorphic if and only if their linearizations belong to either the same GIT chamber or wall. In particular, there is a finite number of non-isomorphic GIT quotients $\overline{M}^{GIT}_{n,d,t}$ corresponding to values $t=t_i$ and to any $t\in (t_i,t_{i+1})$. 

\begin{theorem}
\label{theorem:wall-stratification}
Let $S_{n,d}$ be the set of one-parameter subgroups in Definition \ref{definition:set-S}. All GIT walls  $\{t_0,\ldots, t_{n,d}\}$  correspond to a subset of the finite set
\begin{align}
\label{eq:walls}
\left\{\left. -\frac{\langle m,\lambda\rangle}{\langle x_i, \lambda\rangle} \ \right\vert \ m \text{ is a monomial of degree }d,\ 0\leqslant i\leqslant n+1,\ \lambda\in S_{n,d} \right\}
\end{align}
and they are contained in the interval $[0, t_{n,d}]$ where $t_{n,d}=\frac{d}{n+1}$. Every pair $(X,H)$ has an \emph{interval of stability} $[a,b]$ with $a,b\in\{t_0,\ldots,t_{n,d}\}$. Namely, $(X,H)$ is $t$-semistable if and only if $t\in [t_i,t_j]$ for some walls $t_i,t_j$. If $(X,H)$ is $t$-stable for some $t$ then $(X,H)$ is $t$-stable if and only if $t\in (t_i,t_j)$. 
\end{theorem}

\begin{corollary}
\label{corollary:dimension-moduli}
Assume that the ground field is algebraically closed with characteristic $0$ and that the locus of stable points is not empty and $d \geqslant 3$. Then
$$\dim \overline{M}^{GIT}_{n,d,t}={\binom{n+d+1}{d}}-n^2-3n-3.$$
Each $\overline{M}^{GIT}_{n,d,t}$ is a compactification of the space of log smooth pairs $(X,X\cap H)$ described above.
\end{corollary}

When $X$ is Fano or Calabi-Yau the pairs $(X,H)$ are realized as log pairs:
\begin{theorem}
\label{theorem:CY-Fano}
Every point in the GIT quotient $\overline{M}^{GIT}_{n,d,t}$ parametrizes a closed orbit associated to a pair  $(X,D)$ with $D=X \cap H$ in the cases where $X$ is a Calabi-Yau or a Fano hypersurface of degree $d>1$. Furthermore, if $X$ is Fano $t\leqslant t_{n,d}$ and $(X,D)$ is $t$-semistable, then $X$ does not contain a hyperplane in its support, unless $t=t_{n,d}$, in which case $(X,D)$ is strictly $t_{n,d}$-semistable.
\end{theorem}

Once a set of coordinates is fixed, any pair $(X,H)$ can be determined by homogeneous polynomials $F$ and $F'$ of degrees $d$ and $1$, respectively, which define a pair of sets of monomials, namely those which appear with non-zero coefficients in $F$ and $F'$. Suppose $(X,H)$ is not $t$-stable. We find sets of monomials  $N^\oplus_t(\lambda_k, x_i)$ such that in some coordinate system the equations of $F$ and $F'$ are given by monomials in $N^\oplus_t(\lambda_k, x_i)$. A similar procedure follows for $t$-unstable pairs, where the relevant sets of monomials are $N^+_t(\lambda_k, x_i)$.

\begin{theorem}
\label{theorem:main}
Let $t\in (0,t_{n,d})$.  A pair $(X,H)$ is not $t$-stable ($t$-unstable, respectively), if and only if there exists $g \in SL(n, \mathbb{K})$ such that the set of monomials associated to  $(g\cdot X, g \cdot H)$ is contained in a pair of sets $N^\oplus_t(\lambda, x_i)$ ($N^+_t(\lambda, x_i)$, respectively) defined in Lemma \ref{lemma:max-des-pairs}.

Furthermore, the sets $N^\oplus_t(\lambda, x_i)$ and $N^+_t(\lambda, x_i)$  which are \emph{maximal} with respect to the containment order of sets define families of non-$t$-stable pairs ($t$-unstable pairs, respectively) in $\mathcal R_{n,d}$. 
Any not $t$-stable (respectively $t$-un\-stable) pair $(g\cdot X, g\cdot H)$ belongs to one of these families for some group element $g$.
\end{theorem}
These results allow us to identify non-$t$-stable pairs and these are either strictly $t$-semistable or $t$-unstable.

The Centroid Criterion gives a polyhedral interpretation of stability. Indeed, a pair $(X,H)$ determines a convex polytope $\overline{\mathrm{Conv}_t(X,H)}$ and the parameter $t$ determines a point $\mathcal O_t$ in affine space (for details see Section \ref{sec:centroid}).
\begin{lemma}[Centroid Criterion]
\label{lemma:centroid-criterion}
Let $t\in \mathbb Q_{\geqslant 0}$. A pair $(X,H)$ is $t$-semi\-stable (respectively $t$-stable) if and only if $\mathcal O_t\in \overline{\mathrm{Conv}_t(X,H)}$ ($\mathcal O_t\in \mathrm{Int}\left(\mathrm{Conv}_t(X,H)\right)$, respectively).
\end{lemma}

The boundary of $\overline{M}^{GIT}_{n,d,t}$  is of special interest for GIT problems. Each of the its points has a one-to-one correspondence to a strictly $t$-semistable closed orbit.

\begin{theorem}
\label{theorem:closed-orbits}
Assume $t\in (0,t_{n,d})$ and ground field of characteristic $0$. If a pair $(X,H)$ belongs to a closed strictly $t$-semistable orbit, then there are $g\in \SL_{n+2}$, $\lambda\in S_{n,d}$ and $x_i$ such that the set of monomials associated to $(g\cdot \mathcal X, g\cdot \mathcal H)$ corresponds to those in a pair of sets $(V^0_t(\lambda,x_i),B^0(\lambda,x_i))$ defined such that $(v,b)\in V^0_t(\lambda,x_i)\times B^0(\lambda,x_i)$ if and only if $\mu_t(v,b)=0$ and $(v,b)\in N^\oplus_t(\lambda, x_i)$. 
\end{theorem}

\subsection{Conventions and notation}
\label{sec:notation}
We work over an algebraically closed field $\mathbb K$. Let $G=\mathrm{SL}(n+2,\mathbb K)$ and $T\subset G$ be a fixed maximal torus. The torus $T\cong (\mathbb K^*)^{n+2}$ induces lattices of characters
$M=\mathrm{Hom}_{\mathbb Z}(T,\mathbb G_m)\cong \mathbb Z^{n+2}$
and of one-parameter subgroups
$N=\mathrm{Hom}_{\mathbb Z}(\mathbb G_m, T) \cong \mathbb Z^{n+2},$ with a natural pairing:
$$\langle-,-\rangle\colon M\times N \longrightarrow \mathrm{Hom}_{\mathbb Z}(\mathbb G_m,\mathbb G_m	)\cong \mathbb Z.$$
We choose projective coordinates $(x_0:\cdots : x_{n+1})$ in $\mathbb P^{n+1}$ such that $T$ is diagonal.  Given a one-parameter subgroup $\lambda \colon \mathbb G_m\cong \mathbb K^*\rightarrow T\subset G$ in $M$, we say it is \emph{normalized} (\cite[\S7.2(b)]{mukai-book}) if
$$\lambda(s)=\diag(s^{r_0},\ldots,s^{r_{n+1}})\coloneqq\left( \begin{array}{ccc}
s^{r_0} & \cdots & 0 \\
\vdots & \ddots & \vdots\\
0 & \cdots & s^{r_{n+1}}\end{array} \right),$$
such that $r_0\geqslant \cdots\geqslant r_{n+1}$, $\sum r_i=0$ and not all $r_i=0$. In particular $r_0>0$ and $r_{n+1}<0$.
Denote by $\Xi_k$ the set of all monomials of degree $k$ in variables $x_0,\dots,x_{n+1}$. Since each monomial in $\Xi_k$ can be identified with a character $\mathcal X^a\in M$ of weight $k$, we can see the pairing $\langle-,-\rangle$ of one-parameter subgroups with monomials as: 
$$\langle x_0^{d_0}\cdots x_{n+1}^{d_{n+1}},\diag(s^{r_0},\ldots,s^{r_{n+1}})\rangle=\langle \mathcal X^a, \lambda\rangle=\sum d_i\cdot r_i\in \mathbb Z,$$
where $a=(d_0,\ldots,d_{n+1})\in (\mathbb Z_{\geqslant 0})^{n+2}$, $\sum d_i=k$.

Let $X$ be a hypersurface of degree $d$ defined by  polynomials $F=\sum c_Ix^I$ with $I=(d_0,\ldots,d_{n+1})$ and let $H$ be a hyperplane defined by $\sum h_i x_i$ where $c_I,h_i \in \mathbb K$.  We define their \emph{associated sets of monomials} $(\mathcal X, \mathcal H)$ as the pair of sets:
\begin{align*}
\mathcal X = \{x^I\in \Xi_d \ | \ c_I \neq 0 \}, & & \mathcal H = \{x_i  \in \Xi_1 \ | \ h_i \neq 0 \}.
\end{align*}
Let $\lambda$ be a normalized one-parameter subgroup of $G$. By definition \cite[21, p. 81]{mumford-git}, the \emph{Hilbert-Mumford function} is
$$\mu(X,\lambda)\coloneqq\min\{\langle I,\lambda\rangle \ | \ c_I\neq 0\}.$$
Note that for fixed $X$, the function $\mu(X, -)$ is piecewise linear.  Finally, there is a natural partial order on $\Xi_k$ which we call \emph{Mukai order} \cite[Lemma 7.18]{mukai-book}: given $v,m\in \Xi_k$, 
$$v\leqslant m  \ \Longleftrightarrow \langle v,\lambda\rangle\leqslant \langle m,\lambda\rangle,$$
for all normalized one-parameter subgroups $\lambda$. Under this order there is a unique \emph{maximal element} $x_0^k $ and unique \emph{minimal element} $x_{n+1}^k $ in $\Xi_k$. In the special case when $k=1$, the Mukai order is a total one. 

Our results, together with a good knowledge of the singularities of $(X,D\cap H)$ for given $d$ and $n$, are sufficient to describe all the GIT compactifications. A sketch of such an algorithm is discussed in section \ref{section:algorithm}. We refer the reader to \cite{gallardo-jmg-code} for the details and to \cite{gallardo-jmg-vgit-cubic-surfaces} for the case of cubic surfaces and their anticanonical divisors.

\subsection*{Acknowledgments}
Our work is in debt with  R. Laza whose work on curves inspired us to generalize his results to higher dimensions. We thank him, R. Dervan and D. Swinarski for useful discussions.  P. Gallardo is supported by the NSF grant DMS-1344994 of the RTG in Algebra, Algebraic Geometry, and Number Theory, at the University of Georgia. 

Our results have been implemented in a Python software package available in \cite{gallardo-jmg-code}. Detailed algorithms implemented in the software package will appear in \cite{gallardo-jmg-experimental}. The source code, but not the text of this article, is released under a Creative Commons CC BY-SA $4.0$ license. See \cite{gallardo-jmg-code} for details. If you make use of the source code in an academic or commercial context, you should acknowledge this by including a reference or citation to this article.

\section{VGIT Setting}
\label{sec:setting}

Let $\mathcal R=\mathcal R_{n,d}$  be the  parameter scheme of pairs $(X,H)$
 given by 
$$
\mathcal R_{n,d}=\mathbb P(H^0(\mathbb P^{n+1}, \mathcal O_{\mathbb P^{n+1}}(d)))\times \mathbb P(H^0(\mathbb P^{n+1}, \mathcal O_{\mathbb P^{n+1}}(1))))\cong \mathbb P^N\times \mathbb P^{n+1},
$$
where $N= {\binom{n+1+d}{d}}-1$. 
\begin{lemma}
The set of $G$-linearizable line bundles $\Pic^G(\mathcal R)$ is isomorphic to $\mathbb{Z}^2$.   Then a line bundle $\mathcal L \in\Pic^{G}(\mathcal R),$ is ample if and only if 
$$\mathcal L=\mathcal{O}(a,b)\coloneqq\pi_1^*(\mathcal O_{\mathbb P^N}(a))\otimes\pi_2^*(\mathcal O_{\mathbb P^{n+1}}(b))\in\Pic^{G}(\mathcal R),$$
where $\pi_1$ and $\pi_2$ are the natural projections on $\mathbb P^N$ and $\mathbb P^{n+1}$, respectively and $a, b>0$. 
\end{lemma}
\begin{proof}
Let $\pi_1\colon \mathcal R \rightarrow \mathbb P^N,\ \pi_2\colon \mathcal R\rightarrow \mathbb P^{n+1}$ be the natural projections. The action of $G$ on $\Xi_d$ and $\Xi_1$ induces a natural action on $\mathcal R\cong \mathbb P^{N}\times \mathbb P^{n+1}$, which preserves the fibers. Hence we have an action of $G$ on both $\mathbb P^N$ and $\mathbb P^{n+1}$ and $\pi_1,\pi_2$ are morphisms of $G$-varieties. Recall there is an exact sequence (see \cite[Theorem 7.2]{dolgachev-lectures-git}):
$$0\longrightarrow \mathcal X(G)\longrightarrow \Pic^G(\mathcal R)\longrightarrow\Pic(\mathcal R)\longrightarrow \Pic(G),$$
 where $\mathcal X(G)$ is the kernel of the forgetful morphism $\Pic^G(\mathcal R)\rightarrow \Pic(\mathcal R)$.
Since $\mathcal X(G)=\{1\}$ and $\Pic(G)=\{1\}$ by \cite[Chapter 7.2]{dolgachev-lectures-git} then $\Pic^G(\mathcal R)\cong \Pic(\mathcal R)$. Moreover, given that $\Pic^G(\mathcal R)\subseteq\Pic(\mathcal R)^G\subset \Pic(\mathcal R)$, were $\Pic(\mathcal R)^G$ is the group of $G$-invariant line bundles, there result follows from
$$\Pic^{G}(\mathcal R)\cong\Pic(\mathcal R)^G\cong\Pic(\mathcal R)\cong \pi_1^*(\Pic(\mathbb P^N))\times\pi_2^*(\Pic(\mathbb P^{n+1}))\cong \mathbb Z \times \mathbb Z.$$
\end{proof}
\noindent
For $\mathcal L\cong\mathcal O(a,b)$,  the  GIT quotient is defined as:
$$\overline{M}^{GIT}_{n,d,t}=\Proj\bigoplus_{m\geqslant 0} H^0(\mathcal R,\mathcal L^{\otimes m})^{G},$$
where $t=\frac{b}{a}$. 

The main tool to understand variations of GIT  from a computational viewpoint is the Hilbert-Mumford numerical criterion which in our particular case has the following form.
\begin{lemma}\label{eq:HM-function-pairs}
Given an ample $\mathcal L\cong \mathcal O(a,b) \in \Pic^{G}(\mathcal R)$,  let $(X,H)$ be a pair parametrized by $\mathcal R$, and let $\lambda$ be a normalized one-parameter subgroup of $G$.  The  \emph{Hilbert-Mumford function} (see  \cite[Definition 2.2]{mumford-git}), 
is $\mu^{\mathcal L}((X,H),\lambda) =a \mu_t(X,H,\lambda)$ where $t=\frac{b}{a}\in \mathbb Q_{>0}$ and
\begin{align*}
\mu_t(X,H,\lambda) &\coloneqq\mu(X,\lambda)+t\mu(H,\lambda)\\
&= \min\{\langle I,\lambda\rangle \ | \ x^I  \in \mathcal X \} + t\min\{ r_i \ | \ x_i \in \mathcal H \}.
\end{align*}
\end{lemma}
\begin{proof}
 By \cite[p. 49]{mumford-git}, for fixed $(X,H)$ and $\lambda$, $\mu^{\mathcal L}\colon \Pic^G(\mathcal R)\rightarrow \mathbb Z$ is a group homomorphism. Moreover, given any $G$-equivariant morphism of $G$-varieties $\pi\colon \mathcal R\rightarrow Y$, we have that $\mu^{\pi^*\mathcal L}((X,H),\lambda)=\mu^{\mathcal L}(\pi(X,H),\lambda)$. Applying these two properties, the result follows from:
\begin{align*}
\mu^{\mathcal O(a,b)}((X,H),\lambda)&=\mu^{\pi^*_1\mathcal O_{\mathbb P^N}(a)\otimes\pi^*_2\mathcal O_{\mathbb P^{n+1}}(b)}((X,H),\lambda)\\
&=\mu^{\pi^*_1\mathcal O_{\mathbb P^N}(a)}((X,H),\lambda) + \mu^{\pi^*_2\mathcal O_{\mathbb P^{n+1}}(b)}((X,H),\lambda)\\
&=a\mu^{\mathcal O_{\mathbb P^N}(1)}(X,\lambda) + b\mu^{\mathcal O_{\mathbb P^{n+1}}(1)}(H,\lambda)=a\mu_t(X,H,\lambda).
\end{align*}
\end{proof}
\begin{remark}
\label{remark:stability-monomial-dependent}
Let $(X,H)$ and $(X',H')$ be such that $(\mathcal X,\mathcal H')=(\mathcal X, \mathcal H)$.
Then, $\mu_t(X,H,\lambda) =\mu_t(X',H',\lambda)$.
\end{remark}
\begin{definition}
Let $t\in \mathbb Q_{\geqslant 0}$. 
The pair $(X,H)$ is \emph{$t$-stable} (resp. \emph{$t$-semistable}) if $\mu_t(X,H,\lambda)<0$ (resp. $\mu_t(X,H,\lambda)\leq 0$) for all non-trivial one-parameter subgroups $\lambda$ of $G$. A pair $(X,H)$ is \emph{$t$-unstable} if it is not $t$-semistable. A pair $(X,H)$ is \emph{strictly $t$-semistable} if it is $t$-semistable but not $t$-stable. 
\end{definition}

\section{Stratification of the space of stability conditions}
\label{sec:1PS-and-chambers}

In this section, we fix a maximal torus $T$ of one-parameter subgroups of $G$ and a coordinate system of $\mathbb P^n$ such that $T$ is diagonal in $G$.
\begin{definition}
\label{definition:set-S}
The \emph{fundamental set} $S_{n,d}$ \emph{of one-parameter subgroups} 
$\lambda \in T$ consists of all non-trivial elements $\lambda=\diag(s^{r_0},\ldots,s^{r_{n+1}})$ where
$$(r_0,\ldots,r_{n+1})=c(\gamma_0,\ldots,\gamma_{n+1})\in \mathbb Z^{n+1}$$
satisfying the following:
\begin{itemize}
\item[(1)] $\gamma_i=\frac{a_i}{b_i}\in \mathbb Q$ such that $\gcd(a_i,b_i)=1$ for all $i=0,\ldots,n+1$ and $c=\lcm(b_0,\ldots,b_{n+1})$.
\item[(2)] $1=\gamma_0\geqslant \gamma_1\geqslant \cdots\geqslant \gamma_{n+1}=-1-\sum_{i=1}^n\gamma_i.$
\item[(3)] $(\gamma_0,\ldots,\gamma_{n+1})$ is the unique solution of a consistent linear system given by $n$ equations chosen from the union of the following sets:
\begin{align}
&\mathrm{Eq}(n,d)\coloneqq\left\{\gamma_i-\gamma_{i+1} =0 \ | \ i=0,\ldots,n\right\}\cup\label{eq:Equations-S}\\
&\left\{\sum_{i=0}^{n+1}(d_i-\bar d_i)\gamma_i = 0 \ | \ d_i,\bar d_i\in \mathbb Z_{\geqslant 0} \text{ for all } i \text{ and } \sum_{i=0}^{n+1} d_i=\sum_{i=0}^{n+1} \bar d_i = d\right\}.\nonumber
\end{align}
\end{itemize}
The set $S_{n,d}$ is finite since there are a finite number of monomials in $\Xi_d$.
\end{definition}

\begin{lemma}
\label{lemma:finiteness-lemma}
A pair $(X,H)$ is not $t$-stable (respectively not $t$-semistable) if and only if there is $g\in G$ satisfying
$$\mu_t(X,H) = \max_{\substack{\lambda\in S_{n,d}}}\{\mu_t(g \cdot X, g \cdot H,\lambda)\} \geqslant 0 \qquad (\text{respectively }>0)$$
where $\mu_t$ is the Hilbert-Mumford function defined in Lemma  \ref{eq:HM-function-pairs} and $S_{n,d}$ is the fundamental set of Definition \ref{definition:set-S}.
\end{lemma}
\begin{proof}
Let $R^{ns}_{T_t}$ be the non-$t$-stable loci of $\mathcal R$ with respect to a maximal torus $T$; and let 
$\mathcal R^{ns}$ be the non t-stable loci of $\mathcal R$.
By \cite[p. 137]{dolgachev-lectures-git}), the following holds
\begin{align*}
\mathcal R^{ns} = \bigcup_{T_i \subset G} R_{T_i}^{ns}.
\end{align*}
Let  $(X',H')$ be a pair which is not $t$-stable. Then, $\mu_t(X',H',\rho)\geqslant 0$ for some $\rho \in T'$ in a maximal torus $T'$ which may be \emph{different} from $T$. All the maximal tori are conjugate to each other in $G$, and by \cite[Exercise 9.2.(i)]{dolgachev-lectures-git} the following holds:
$$
\mu_t((X',H'), \rho) = \mu_t(g \cdot (X',H'), g\rho g^{-1}) .
$$
Then, there is $g_0 \in G$ such that $\lambda:=g_0\rho g_0^{-1}$ is normalized and $(X,H)\coloneqq g_0\cdot(X',H')$ has coordinates in our coordinate system such that $\mu_t(X,H,\lambda)\geqslant 0$.  In these coordinate system one-parameter subgroups form a closed convex polyhedral subset $\Delta$ of dimension $n+1$ in $M\otimes \mathbb Q\cong \mathbb Q^{n+2}$ (in fact $\Delta$ is a standard simplex). Indeed, this is the case since for any normalized one-parameter subgroup, $\lambda=\diag(s^{r_0},\ldots,s^{r_{n+1}})$, $\sum r_i=0$ and $r_i-r_{i+1}\geqslant 0$. 

By Lemma \ref{eq:HM-function-pairs}, for any fixed $t$, $X$ and $H$, the function $\mu_t(X,H,-):M\otimes \mathbb Q\rightarrow \mathbb Q$ is piecewise linear. The critical points of $\mu_t$ (i.e. the points where $\mu_t$ fails to be linear) correspond to those points in $M\otimes\mathbb Q$ where $\langle I, \lambda\rangle = \langle\overline{I},\lambda\rangle$ for $I=(d_0,\ldots,d_{n+1}),\ \overline{I}=(\overline{d_0},\ldots,\overline{d_{n+1}})$ representing monomials of degree $d$ of the form $f=\sum f_Ix^I$ defining $f$. Since $\langle-,-\rangle$ is bilinear that is equivalent to say that $\langle I-\overline{I},\lambda\rangle=0$. These points define a hyperplane in $M\otimes \mathbb Q$ and the intersection of this hyperplane with $\Delta$ is a simplex $\Delta_{I,\overline{I}}$ of dimension $n$.

The function $\mu_t(X,H,-)$ is linear on the complement of the hyperplanes defined by $\langle I-\overline{I},\lambda\rangle=0$. Hence its minimum is achieved on the boundary, i.e. either on $\partial\Delta$ or on $\Delta_{I,\overline{I}}$ which are all convex polytopes of dimension $n$. We can repeat this reasoning by finite inverse induction on the dimension until we conclude that the minimum of $\mu_t(X,H,-)$ is achieved at one of the vertices of $\Delta$ or $\Delta_{I,\overline{I}}$. But these correspond precisely, up to multiplication by a constant, to the finite set of one-parameter subgroups in $S_{n,d}$.
\end{proof}
\begin{corollary}
\label{corollary:interval-stability}
Let $(X,H)\in \mathcal R$ and
\begin{align*}
a&=\min\{t\in \mathbb Q_{\geq 0}\ | \ \mu_t(g\cdot (X,H),\lambda_i) \leq 0 \text{ for all } \lambda_i\in S_{n,d},\ g\in G\},\\
b&=\max\{t\in \mathbb Q_{\geq 0}\ | \ \mu_t(g\cdot (X,H),\lambda_i) \leq 0 \text{ for all } \lambda_i\in S_{n,d},\ g\in G\}.
\end{align*}
If $(X,H)$ is $t$-semistable for some $t\in \mathbb Q_{\geq 0}$, then
\begin{itemize}
	\item[(i)] $(X,H)$ is $t$-semistable if and only if $t\in[a,b]\cap \mathbb Q_{\geq 0}$,
	\item[(ii)] if $(X,H)$ is $t$-stable for some $t$, then $(X,H)$ is $t$-stable for all $t\in (a,b)\cap \mathbb Q_{>0}$.
\end{itemize}
We will call $[a,b]$ the \emph{interval of stability} of the pair $(X,H)$. We say $[a,b]=\emptyset$ if $(X,H)$ is $t$-unstable for all $t\in \mathbb Q_{\geq 0}$.
\end{corollary}
\begin{proof}
Recall that $S_{n,d}$ is a finite set, by Lemma \ref{lemma:finiteness-lemma}. Moreover, the pair $(X,H)$ is $t$-(semi)stable if and only if
$$ 
\mu_t(X,H)=\max_{\substack{\lambda_i\in S_{n,d}\\ g\in G}}\{\mu_t(g\cdot(X,H),\lambda_i)\} 
<  0 
\;\;  \left( \leq 0 \right)
$$
Notice that each of the functions $\mu_t(g\cdot (X,H),\lambda_i)$ is affine on $t$ and that there are only a finite number of such functions to consider in the definition of $\mu_t(X,H)$. Indeed, the last statement follows from observing that $\mu_t$ depends only of $\lambda_i$ (finite number of choices in $S_{n,d}$ to consider) and the monomials with non-zero coefficients in the polynomials defining $g\cdot (X,H)$. But there are only a finite number of such subsets of those monomials, since $\mathcal P(\Xi_d)\times \mathcal P(\Xi_1)$ is finite.

To see that $b<\infty$, observe that any hyperplane in $\mathbb P^{n+1}$ is conjugate by an element of $G$ to the hyperplane given by $\{x_0=0\}$. Let $r=(1,0,\ldots,0,-1)\in \mathbb Z^{n+2}$ and $\lambda=\diag(s^r)\in S_{n,d}$. Then $\mu(\{x_0=0\},\lambda)=1>0$. Hence, for $t\gg 0$, we have that $\mu_t(X,D)>0$ as each $\mu_t(g\cdot(X,D),\lambda)$ is piecewise affine. We conclude that if $(X,D)$ is not $t$-semistable for some $t\in \mathbb Q_{\geq 0}$, then 
$$[a,b]=\bigcap_{\substack{\lambda_i\in S_{n,d}\\g\in G}}\{t\ | \ \mu_t(g\cdot (X,H), \lambda_i)\leq 0\}$$
is a bounded interval, as it is an intersection of a finite number of intervals. This proves (i).

For (ii), notice that $(X,H)$ being $t$-stable for some  $t_0$ is equivalent to the functions $\mu_{t_0}(g\cdot (X,H),\lambda_i)$ being  always strictly negative. Then, the statement follows because  $\mu_{t}(g\cdot (X,H),\lambda_i)$ are affine functions, and $[a,b]$ is a compact interval.
\end{proof}

\section{Centroid Criterion}
\label{sec:centroid}

Lemma \ref{lemma:finiteness-lemma} allows us to detect the lack of $t$-stability of a $G$-orbit by having to consider only a finite number of one-parameter subgroups, precisely those in $S_{n,d}$. However, sometimes it is convenient to decide on the $t$-stability of a given pair $(X,H)$ without comparing to all the elements in $S_{n,d}$. For this purpose and to shorten the proof of Theorem \ref{theorem:wall-stratification}, we developed the Centroid Criterion, for which we need to introduce extra notation. Fix $t\in \mathbb Q_{\geqslant 0}$. We have a map $\mathrm{disc}_t\colon \Xi_d\times \Xi_1\rightarrow M\otimes \mathbb Q \cong \mathbb Q^{n+2}$, defined as
$$
\mathrm{disc}_t(x_0^{d_0}\cdots x_{n+1}^{d_{n+1}},x_j)
=(d_0,\ldots,d_{j-1}, d_j+t, d_{j+1} \dots d_{n+1}).$$
The image of $\mathrm{disc}_t$ is supported on the first quadrant of the hyperplane
$$H_{n,d,t}=\left\{(y_0,\ldots,y_{n+1})\in\mathbb Q^{n+2} \Bigg| \ \sum_{i=0}^{n+1} y_i=d+t\right\}.$$

We define the set $\mathrm{Conv}_t(X,H)$ as the convex hull of
$$\{\mathrm{disc}_t(v,b) \ | \ v\in \mathcal X,\ b=\min(\mathcal H)\}\subset H_{n,d,t},$$
where the minimum is for the Mukai order in $\Xi_1$, which is a total order (see Section \ref{sec:notation}). Observe that $\overline{\mathrm{Conv}_t(X,H)}$ is a convex polytope.

Given $t\in \mathbb Q_{\geqslant 0}$, we define the \emph{$t$-centroid} as
$$\mathcal O_t=\mathcal O_{n,d,t}=\left(\frac{d+t}{n+2},\ldots,\frac{d+t}{n+2}\right)\in H_{n,d,t}\subset\mathbb Q^{n+2}.$$ 

\begin{proof}[Proof of Lemma \ref{lemma:centroid-criterion}]
First we note that $(X,H)$ is $t$-semistable ($t$-stable, respectively) if and only if $(X,X\cap\{\min(\mathcal H)=0\})$ is $t$-semistable ($t$-stable, respectively). Indeed, let $x_k=\min(\mathcal H)$. Given any one-parameter subgroup $\lambda\in N$ we have
\begin{align*}
\mu_t((X,H),\lambda)&=\mu(\mathcal X,\lambda) + t\min_{b\in \mathcal H}\{\langle b,\lambda\rangle\}\\
&=\mu(\mathcal X,\lambda)+t\langle x_k,\lambda\rangle =\mu_t((X,X\cap \{x_k=0\}),\lambda).
\end{align*}
Hence, we may assume $D=X\cap \{x_k=0\}$. Suppose $\mathcal O_t\not\in \overline{\mathrm{Conv}_t(X,H)}$, then there is an affine function $\phi\colon \mathbb R^{n+2}\rightarrow \mathbb R$ such that $\phi\vert_{\overline{\mathrm{Conv}_t(X,H)}}$ is positive and $\phi(\mathcal O_t)=0$. In fact, since the vertices of $\mathrm{Conv}_t(X,H)$ have rational coefficients, we can choose $\phi$ to have integral coefficients. Write
$$\phi(y_0,\ldots,y_{n+1})=\sum_{i=0}^{n+1}a_iy_i + l$$
For $\mathrm{disc}_t(x^{d_0}\cdots x^{d_{n+1}},x_k)=(d_0,\ldots,d_k+t,\ldots,d_{n+1})\in \overline{\mathrm{Conv}_t(X,H)}$ we have
$$\sum_{i=0}^{n+1}a_id_i + ta_k + l>0,$$
and since $\phi(\mathcal O_t)=0$, we obtain $\frac{d+t}{n+2}\sum_{i=0}^{n+1}a_i+l=0$. Let $p=-\frac{l}{d+t}\in \mathbb Q$ and choose $m\in \mathbb Z_{\geqslant 0}$ such that $mp\in \mathbb Z$. Let
$$\lambda(s)=\left( \begin{array}{ccc}
s^{m(a_0-p)} & \cdots & 0 \\
\vdots & \ddots & \vdots\\
0 & \cdots & s^{m(a_{n+1}-p)}\end{array} \right)\in N.$$
Hence
\begin{align*}
\mu_t((X,H),\lambda)&=\min_{\prod_{i}x_i^{d_i}\in \mathcal X}\left\{\sum_{i=0}^{n+1}m(a_i-p)d_i\right\} + tm(a_k-p)\\
&=m\left(\min_{\prod_{i}x_i^{d_i}\in \mathcal X}\left\{\left(\sum_{i=0}^{n+1}a_id_i\right)+ta_k\right\}-p(d+t)\right)\\
&=\min_{v\in\overline{\mathrm{Conv}_t(X,H)}}\phi(v)>0.
\end{align*}
Hence $(X,H)$ is not $t$-semistable. We have shown that if $(X,H)$ is $t$-semistable, then $\mathcal O_t\in \overline{\mathrm{Conv}_t(X,H)}$. The proof of the statement when $(X,H)$ is $t$-stable is similar; in the above reasoning we only need to swap $\overline{\mathrm{Conv}_t(X,H)}$ by $\mathrm{Int}({\mathrm{Conv}_t(X,H)})$ and the strict inequalities by $\geqslant 0$.

Conversely, suppose that $(X,H)$ is not $t$-semistable. Then there is a normalized one-parameter subgroup $\lambda=\diag(s^{r_0},\ldots,s^{r_{n+1}})\in N$ with $\sum r_i=0$ and such that
\begin{align*}
0<\mu_t(X,H,\lambda)&=\min_{\prod_{i}x_i^{d_i}\in \mathcal X}\left\{\sum_{i=0}^{n+1}d_ir_i\right\} + tr_k\\
&=\min_{\prod_{i}x_i^{d_i}\in \mathcal X}\left\{r_0d_0+\cdots+r_k(d_k+t)+\cdots r_{n+1}d_{n+1}\right\}.
\end{align*}
Let $\phi(y_0,\ldots,y_{n+1})=\sum r_i y_i$.	By convexity, we have that $\phi\vert_{\overline{\mathrm{Conv}_t(X,H)}}>0$. On the other hand $\phi(\mathcal O_t)=\sum r_i\frac{d+t}{n+2}=\frac{d+t}{n+2}\sum r_i=0$. Hence $\mathcal O_t\not\in \overline{\mathrm{Conv}_t(X,H)}$. The proof for $t$-stability is similar.
\end{proof}

\begin{lemma}
\label{lemma:interval-stability2}
Let $(X,H)\in \mathcal R$. Suppose that its interval of semi-stability $[a,b]$ is not empty. Then
\begin{itemize}
	\item[(i)] $a=0$ if and only if $X$ is a GIT-semistable hypersurface of degree $d$.
	\item[(ii)] $b\leq t_{n,d}=\frac{d}{n+1}$ 
         \item[(iii)] The pair $(X,H)$ is $t_{n,d}$-semistable if and only if $X\cap H$ is a semistable hypersurface of degree $d$ in $H\cong \mathbb P^n$. 
\end{itemize}
\end{lemma}
\begin{proof}
The first statement  holds because the Hilbert-Mumford function at $t=0$ coincides with the Hilbert-Mumford function  for hypersurfaces, and the natural projection $\mathcal R\rightarrow \mathbb P^{n+1}$ is $G$-invariant.

For part (ii), suppose that $t>t_{n,d}$. Without loss of generality, we can suppose that the equations of any pair $(X,H)$ are given by 
\begin{align*}
X=\left( F(x_0,  \ldots, x_{n+1}) =0 \right),
& &
H=(x_0=0).
\end{align*}
Let $\lambda=(n+1, -1, \ldots, -1)$, then
$$
\mu_{t_{n,d}}((X,H), \lambda)>
-d+\frac{d}{n+1}(n+1) =0
$$
holds and $(X,H)$ is $t$-unstable. Therefore $b\leqslant t_{n,d}$. 

Next, we discuss (iii).  Suppose that $Y_0:=X \cap H$ is unstable.  We can select a coordinate system such that 
\begin{align*}
Y_0:= \left( p_d(x_1, \ldots ,x_{n+1}) =0 \right), \qquad H:=\{x_{0}=0\},
\end{align*}
and $Y_0$ is unstable in the coordinate system $\{x_1, \ldots, x_{n+1} \}$.  By using Mukai's order of monomials, we claim that 
among all possible pairs $(X,H)$ such that $Y_0=X \cap H$, the pair $(\tilde X,H)$ given by
\begin{align}\label{eq:mukaiMin}
\tilde X:=p_d(x_1, \ldots ,x_{n+1})+x_{n+1}^{d-1}x_0,
 &  &
H:=(x_0=0),
\end{align}
will minimize the Hilbert-Mumford function for any normalized one-parameter subgroup, because
$$
\mu(\tilde X, \lambda)
=\min\{
\mu(X, \lambda) \; | \; X  \cap H=Y_0
\}.
$$ 
Indeed, any other  $X$ with $X \cap H=Y_0$ differs from $\tilde X$ by a monomial involving the variable $x_0$.  Then, we observe that any other monomial divided by $x_0$  is greater than $x_0x_{n+1}^{d-1}$ in Mukai's order.
As a consequence if $\mu_{t_{n,d}}(\tilde X, H, \lambda) >0$, then any pair $(X,H)$ with $X \cap H=Y_0$ is $t_{n,d}$-unstable. Next, we use the Centroid Criterion (Lemma \ref{lemma:centroid-criterion}). 
By hypothesis, $Y_0$ is not a semistable hypersurface in $\mathbb{P}^n$. Then the convex hull of its monomials does not contain the point $\left( \frac{d}{n+1}, \ldots, \frac{d}{n+1} \right) \in \mathbb{R}^{n+1}$.  
If we consider the monomials of $Y_0$ as monomials  in $\mathbb{K}[x_0, \ldots, x_{n+1}]$, then the convex hull of the monomials of $\tilde{X}$ does not contain the point $\left(0, \frac{d}{n+1}, \ldots, \frac{d}{n+1} \right) \in \mathbb{R}^{n+2}$.
Notice  that this implies that $ \overline{\mathrm{Conv}_{t_{n,d}}(\tilde X,(x_0=0))}$   does not contain the point
$$\mathcal O_{n,d,t_{n,d}}=\left(0, \frac{d}{n+1}, \ldots, \frac{d}{n+1} \right)+ \frac{d}{n+1}(1, 0  , \ldots, 0),$$
and by the Centroid Criterion $(X,H)$ is $t_{n,d}$-unstable. To see the last assertion, notice that $ \overline{\mathrm{Conv}_{t_{n,d}}(\tilde X,(x_0=0))}$ is the convex hull of
$$V=\{\mathrm{disc}_{t_{n,d}}(m, x_0)\ | \ m \text{ is a monomial in }p_d\}\subset P\coloneqq\{y_0=t_{n,d}\}\subset\mathbb R^{n+2},
$$
and the point $q\coloneqq(1+t_{n,d},\ldots, 0,d-1)\not\in P$. Therefore $ \overline{\mathrm{Conv}_{t_{n,d}}(\tilde X,(x_0=0))}$ is a pyramid with base $V$ and vertex $q$. Since $\mathcal O_{t_{n,d}}\in P\setminus V$, the claim follows.

Next suppose that $(X,H)$ is unstable.  Then,  by the Centroid Criterion there is a coordinate system such that 
$ \overline{\mathrm{Conv}_t(X,H)}$ does not contain the centroid $\mathcal{O}_{n,d, \frac{d}{n+1}}$.  By using the Mukai order as in the previous case, we may assume $H=\{x_i=0\}$. Let
$$v:=\left(\frac{d}{n+1}, \ldots, \frac{d}{n+1}, 0, \frac{d}{n+1}, \ldots, \frac{d}{n+1} \right),$$
where the value $0$ corresponds to the $i$-th entry. Observe $v\not\in \overline{\mathrm{Conv}_0(X,H)}$, since otherwise
$$\mathcal O_{n,d,t_{n,d}}=\mathrm{disc}_{\frac{d}{n+1}}(v,x_i) \in  \overline{\mathrm{Conv}_{t_{n,d}}(X,H)}.$$
The monomials in the polynomial defining $X \cap (x_i=0)$ are precisely those monomials $m_j$ in the polynomial defining $X$ with exponents of the form
$$a_j=(d_0^j,\ldots, d_{i-1}^j, 0, d_{i+1}^j, \ldots d_{n+1}^j).$$
Those monomials correspond to the points generating a face $F$ of $\overline{\mathrm{Conv}_{t_{n,d}}(X,H)}$, namely the convex hull of points $(d_0^j,\ldots, d_{i-1}^j, t_{d,n}, d_{i+1}^j, \ldots d_{n+1}^j)$. The projection $F_P$ of $F$ onto the hyperplane $P=\{y_i=0\}\subset\mathbb R^{n+2}$ gives us that $v\not\in F_P$ since $F_P\subseteq\overline{\mathrm{Conv}_0(X,H)}$. But $F_P$ corresponds to $\overline{\mathrm{Conv}_0(X\cap H)} $ and $v=\mathcal O_{n-1,d,0}$, so by the Centroid Criterion $X\cap H$ is unstable.
 \end{proof}

\begin{proof}[{Proof of Theorem \ref{theorem:wall-stratification}}]
By Remark \ref{remark:stability-monomial-dependent} and the fact $\mathcal P(\Xi_k)\times \mathcal P(\Xi_1)$ is a finite set, there is a finite number of possible intervals of stability, say $[a_j,b_j]$. Hence $t_i\in \bigcup_j\{a_j,b_j\}$ and Lemma \ref{lemma:interval-stability2} implies that all $b_i\leqslant t_{n,d}$. Notice that given any wall $t_i$ there is at least a pair $(X,H)$ such that 
$$\mu_t(X,H)=\max_{\lambda\in S_{n,d}}\{\mu_t((X,H),\lambda)\}$$
satisfies $\mu_t(X,H)\leq0$ for $t\leq t_i$ and $\mu_t(X,H)>0$ for $t>t_i$. Hence $\mu_t(X,H)=0$ for $t=t_i$ since $\mu_t$ is continuous. The result follows from Lemma \ref{lemma:finiteness-lemma} and Remark \ref{remark:stability-monomial-dependent}.
\end{proof}

\begin{proof}[Proof of Corollary \ref{corollary:dimension-moduli}]
From \cite[Theorem 2.1]{orlik-solomon-finite-automorphism-hypersurface}, any hypersurface $X$ of degree $d\geqslant 3$ has $\dim(\mathrm{Aut}(X))=0$. Hence, for any log smooth pair $p=(X,D)\in \mathcal R$, its stabilizer $G_p$ satisfies
$$0\leq \dim(G_p)=\dim(G_X\cap G_D)\leq \dim( G_X)\leq \dim( \mathrm{Aut}(X))=0,$$
where the last equality follows from \cite[Theorem 2.1]{orlik-solomon-finite-automorphism-hypersurface}.
The result follows from the following identity (see \cite[Corollary 6.2]{dolgachev-lectures-git}):
\begin{align*}
\dim\left(\overline{M}^{GIT}_{n,d,t}\right)&=\dim(\mathcal H)-\dim(G)+\min_{p\in \mathcal H}{\dim G_p}=\\
&=\left({\binom{n+1+d}{d}}-1 + (n+1)\right)-\left((n+2)^2-1\right).
\end{align*}
We are left to proof the following claim: any pair $(X,D)$ such that $D\not\subset X$ and $X\cap D$ is smooth is $t$-semistable for all $t\in [0,t_{n,d}]$. Since smooth hypersurfaces are GIT semistable, the claim follows from Lemma \ref{lemma:interval-stability2}.
\end{proof}

\begin{proof}[Proof of Theorem \ref{theorem:CY-Fano}]
Suppose $n+2\geq d$ and $(X,H)$ is a pair such that $ \mathrm{Supp}(H) \subset \mathrm{Supp}(X)$. It suffices to show $(X,H)$ is $t$-unstable for all $t\geqslant 0$. We choose a coordinate system such that $H=(x_0=0)$ and $X$ is given as the zero locus of $F =x_0f_{d-1}(x_0,x_1, \ldots, x_{n+1})$. 
The monomial $x_0x^{d-1}_{n+1}$ is minimal for the Mukai order in the set $\mathcal X\cup \{x^{d-1}_{n+1}x_0\}$. Then for any  normalized one-parameter subgroup $\lambda=(s^{r_0}, \ldots, s^{r_{n+1}})$, 
\begin{align*}
\mu_{t}(X,H, \lambda) \geq \left( r_0+(d-1)r_{n+1} \right)+tr_0 
\end{align*}
holds. Since $d\leqslant n+2$, the one parameter subgroup $\lambda_0=\mathrm{Diag}(s^r)$ is normalized, where
$$
r=
\left(
n(d-1),  -(d-2), \ldots,  -(d-2), -n
\right).
$$
Then $\mu_{t}(X,H, \lambda_0)  \geq tn(d-1) > 0$ and $(X,H)$ is unstable for any $t>0$.
Now, if $X$ is a reducible Fano hypersurface, then $d\leqslant n+1$ and we may assume $F=x_0f_{d-1}(x_0,\ldots,x_{n+1})$. Let $\lambda=\diag(n+1,-1,\ldots,-1)$. By the previous step $H\neq \{x_0=0\}$ so 
$$\mu_t(X,D,\lambda)=n+1-(d-1)-t,$$
	and as $t\leqslant \frac{d}{n+1}$ and $d\leqslant n+1$, the result follows.
\end{proof}

\section{Families of $t$-unstable pairs}
\label{sec:maximal-destabilizing-subsets}

In this section we determine, for a given $t$, the set of monomials that characterize non-$t$-stable and $t$-unstable pairs.
\begin{definition}
Fix  $t\in [0, t_{n,d} ]$, and let $\lambda$ be a normalized one-parameter subgroup. A non empty pair of sets $A\subset  \Xi_d$ and $B\subset \Xi_1$ is a  \emph{maximal $t$-(semi)de\-stab\-i\-lized pair} $(A,B)$ with respect to $\lambda$ if the following conditions hold:
\begin{enumerate}[(i)]
\item Each pair $(v,m) \in A\times B$ satisfies $\langle v, \lambda\rangle + t\langle m , \lambda\rangle>0 $ ($\geqslant 0$, respectively).
\item If there is another pair of sets $\tilde A\subset \Xi_d$, $B\subset \Xi_1$ such that $A\subseteq \tilde A$, $B \subseteq \tilde B$ and for all $(v,m)\in \tilde A\times\tilde B$ the inequality $\langle v, \lambda\rangle + t\langle m , \lambda\rangle>0 $ ($\geqslant 0$, respectively) holds, then $\tilde A=A$ and $\tilde B = B$.
\end{enumerate}  
\end{definition}

\begin{lemma}\label{lemma:max-des-pairs}
Given a one-parameter subgroup $\lambda$ any maximal $t$-(semi)\-destabilized pair with respect to $\lambda$ can be written as
\begin{align*}
N^+_t(\lambda,x_i)&\coloneqq(V^+_t(\lambda, x_i),B^+(x_i))\\
\text{(respectively }N^\oplus_t(\lambda,x_i)&\coloneqq(V^\oplus_t(\lambda, x_i),B^\oplus(x_i))\text{)}
\end{align*}
where $x_i\in \Xi_1$ and
\begin{align*}
V^+_t(\lambda, x_i) &\coloneqq\{v \in \Xi_d \ | \ \langle v, \lambda\rangle + t\langle x_i,\lambda\rangle> 0\},\
B^+(x_i) \coloneqq\{m \in \Xi_1 \ | \ m\geq x_i\},\\
V^\oplus_t(\lambda, x_i) &\coloneqq\{v \in \Xi_d \ | \ \langle v, \lambda\rangle + t\langle x_i,\lambda\rangle\geq 0\},\
B^\oplus(x_i) \coloneqq\{m \in \Xi_1 \ | \ m\geq x_i\}.
\end{align*}
\end{lemma} 
\begin{proof}
Let $(A,B)$ be a maximal  $t$-semidestabilized pair with respect to $\lambda$.
Let $x_i:= \min(B)$. By Mukai's order,
\begin{align*}
\langle v, \lambda \rangle + t\langle m, \lambda\rangle
\geq
\langle v, \lambda \rangle + t\langle x_i , \lambda\rangle\geq 0
, & & \text{ for all } (v,m) \in (A,B).
\end{align*}
Then $\left( A, B \right) \subseteq N^\oplus_t(\lambda,x_i)$ and the maximality condition implies  $(A,B) = N^\oplus_t(\lambda,x_i)$. In particular, this proves that $N^\oplus_t(\lambda,x_i)$ is a maximal $t$-semi\-de\-stab\-i\-lized pair with respect to $\lambda$. The proof for maximal $t$-destabilized pairs is similar, exchanging the inequalities for strict inequalities.
\end{proof}
\begin{proof}[{Proof of Theorem \ref{theorem:main}}]
Suppose   $(X,H)$ is a $t$-unstable pair (a not $t$-stable pair, respectively). By Lemma \ref{lemma:finiteness-lemma} there is $g \in G$ and $\lambda\in S_{n,d}$ such that: 
\begin{align*}
\mu_t(g\cdot X,g \cdot H),\lambda) >0 \text{ (}\geq 0\text{, respectively)}.
\end{align*}
Then, every $(v,m) \in (g \cdot \mathcal X, g \cdot \mathcal H)$ satisfies 
$\langle v, 
\lambda\rangle + t\langle m , \lambda\rangle>0 $ ($\geqslant 0$, respectively).
By the definition of maximal $t$-(semi)stable pairs 
and Lemma \ref{lemma:max-des-pairs},  $g\cdot \mathcal X \subseteq V_t^+(\lambda,x_i)$ and $g\cdot \mathcal H\subseteq B^+(x_i)$ ($g\cdot \mathcal X \subseteq V_t^\oplus(\lambda,x_i)$ and $g\cdot \mathcal H\subseteq B^\oplus(x_i)$, respectively) hold for some $\lambda\in S_{n,d}$ and some $x_i\in \Xi_1$. Choosing the maximal pairs of sets $N_t^\oplus(\lambda, x_i)$ under the containment order where $\lambda\in S_{n,d}$ and $x_i\in \Xi_1$, we obtain families of pairs whose coefficients belong to maximal $t$-(semi)destabilized sets.
\end{proof}

\begin{proposition}
\label{proposition:cartesian-product}
Let $t\in (0, t_{n,d})$. If the set 
\begin{align*}
 \mathrm{Ann}_t \left(\lambda,x_i \right)
=
\{ (v,m) \in V^\oplus_t(\lambda,x_i) \times B^\oplus(x_i)\; | \; \langle v, \lambda \rangle + t\langle m, \lambda\rangle =0 \}
\end{align*}
is not empty, then it is equal to the cartesian product 
$ 
V_t^0 \left(\lambda,x_i \right) \times   B^0 \left(\lambda,x_i \right)
$
where
\begin{align*}
V_t^0 \left(\lambda,x_i \right) &= \{ v \in V^\oplus_t(\lambda, x_i) \; | \exists m' \in B^\oplus(x_i)
\text{ such that }
\langle v, \lambda \rangle + t\langle m', \lambda\rangle =0 \},
\\
B^0 \left(\lambda,x_i \right) &=
\{ m \in B^\oplus(x_i) \; | \;  
\langle \overline m, \lambda \rangle  \geq \langle m, \lambda\rangle
\text{ for all }   \overline m \in B^\oplus(x_i) \}.
\end{align*}
We call $\mathrm{Ann}_t\left(\lambda,x_i \right)$ the \emph{Annihilator of $\lambda$ and $x_i$ at $t$}.
\end{proposition}
\begin{proof}
Let $(v,m) \in \mathrm{Ann}_t \left(\lambda,x_i \right)$. Then
$v \in V^0_t \left(\lambda,x_i \right)$. Suppose there is $\overline m
\in B^\oplus(x_i)$ such
that 
$\langle m, \lambda \rangle  >  \langle \overline m, \lambda\rangle$.  Then, since $t>0$
\begin{align*}
0=\langle v, \lambda \rangle + t\langle m, \lambda\rangle  >
\langle v, \lambda \rangle + t\langle \overline m, \lambda\rangle, 
\end{align*}
which contradicts the fact that $(v,\overline m) \in N^\oplus_t(\lambda,x_i)$. Therefore $m\in B^0(\lambda, x_i)$.

Let $(v,m) \in V^0_t \left(\lambda,x_i \right) \times   B^0 \left(\lambda,x_i \right)$. Then there is $m' \in  B^\oplus(x_i)$ such that $\langle v, \lambda \rangle + t\langle m', \lambda\rangle  =0$. Since $m\in B^0 \left(\lambda,x_i \right)$, then $\langle  m', \lambda \rangle  \geq \langle m, \lambda\rangle$. 
Therefore, we have that 
$$
0=\langle v, \lambda \rangle + t\langle m', \lambda\rangle   \geqslant 
\langle v, \lambda \rangle + t\langle m, \lambda\rangle \geqslant 0,
$$
because  $(v,m) \in N^\oplus_t(\lambda,x_i)$.  This implies that $(v,m) \in \mathrm{Ann}_t(\lambda,x_i)$.
\end{proof}

\begin{proof}[Proof of Theorem \ref{theorem:closed-orbits}]
Let $p=(X,H)$. By \cite[Remark 8.1 (5)]{dolgachev-lectures-git}, since $p=(X,H)$ is strictly $t$-semistable and represents a closed orbit, then the stabilizer subgroup $G_p\subset G=\SL(n+2,\mathbb K)$ is infinite. This implies there is a one-parameter subgroup $\lambda \in G_p$. In particular $\lim_{s\rightarrow 0} \lambda(s)\cdot (X,H)=(X,H)$. This implies $\mu_t(X,H,\lambda)=0$. By choosing an appropriate coordinate system and applying Lemma \ref{lemma:finiteness-lemma} we may assume that $\lambda\in S_{n,d}$ and $(\mathcal X, \mathcal D)=(V^0_t(\lambda, x_i),B^0(\lambda, x_i))$. Indeed, the latter follows from Proposition \ref{proposition:cartesian-product}.
\end{proof}

\section{A method to study stability}
\label{section:algorithm}
The following method can be extended to a full algorithm to describe $M^{GIT}_{n,d,t}$ \cite{gallardo-jmg-code}:
\begin{enumerate}[1.]
	\item By Theorem \ref{theorem:wall-stratification}, the interval of stability of any pair $(X,H)$ for \emph{any} polarization $t\in \mathbb Q_{\geqslant 0}$ is determined by a finite set of one-parameter subgroups $S_{n,d}$ which can be computed using Definition \ref{definition:set-S}. 
	\item The walls $t_0,\ldots,t_{n,d}$ are among those in \eqref{eq:walls}.
	\item For each wall $t=t_i$ or for any $t\in (t_i,t_{i+1})$ we may compute the sets of monomials $N^{\oplus}_t(\lambda, x_j)$ for each $\lambda \in S_{n,d}$ and $0\leqslant j\leqslant n+1$ and choose the \emph{maximal} among them. By Theorem \ref{theorem:main}, these correspond to families in $R_{n,d}$ of non $t$-stable pairs. Each non $t$-stable pair corresponds to one of these families. Then, the Centroid Criterion (Lemma \ref{lemma:centroid-criterion}) distinguishes for which of these families the general element is strictly $t$-semistable or $t$-unstable. 
	\item For each family which is strictly $t$-semistable, we consider the set 
$$
(V^0_t(\lambda,x_i),B^0(\lambda,x_i)) 
$$ 
of each maximal $ N^\oplus_t(\lambda, x_j) $. Any strictly $t$-semistable closed orbit must belong to families whose defining equations have monomials in  this set. (Theorem \ref{theorem:closed-orbits}). These are also called $t$-polystable orbits which are not $t$-stable.
	\item To determine the $t$-stable orbits geometrically, we classify the families given by $N^\oplus_t(\lambda, x_j)$ according to their singularities. This requires an understanding of the singularities of $(X,H)$ for given $n$ and $d$ as well as their deformations. See \cite{gallardo-jmg-vgit-cubic-surfaces} for the case of cubic surfaces.
\end{enumerate}
\bibliographystyle{amsplain}
\bibliography{bibliography} 	

\end{document}